\documentclass[12pt]{amsart}

\newtheorem{theorem}{Theorem}[section]

\newtheorem{proposition}[theorem]{Proposition}
\newtheorem{remark}[theorem]{Remark}

\newtheorem{question}[theorem]{Question}

\newcommand{\ncom}{\newcommand}
\ncom{\ep}{\epsilon}
\ncom{\rar}{\rightarrow}
\ncom{\thrar}{\twoheadrightarrow}
\ncom{\lrar}{\longrightarrow}
\ncom{\ov}{\overline}
\ncom{\what}{\widehat}

\newcommand{\ignore}[1]{}
\ncom{\m}{\mbox}
\ncom{\sta}{\stackrel}
\ncom{\C}{{\mathbb C}}
\ncom{\A}{{\mathbb A}}
\ncom{\Z}{{\mathbb Z}}
\ncom{\Q}{{\mathbb Q}}
\ncom{\R}{{\mathbb R}}
\ncom{\G}{{\mathbb G}}
\ncom{\HH}{{\mathbb H}}
\ncom{\al}{\alpha}
\ncom{\p}{{\mathbb P}}
\ncom{\N}{{\mathbb N}}
\ncom{\K}{{\mathbb K}}
\ncom{\X}{{\mathbb X}}
\ncom{\f}{\frac}
\ncom{\cA}{{\mathcal A}}
\ncom{\cB}{{\mathcal B}}
\ncom{\cD}{{\mathcal D}}
\ncom{\cDB}{{\mathcal D \mathcal B}}
\ncom{\cX}{{\mathcal X}}
\ncom{\cO}{{\mathcal O}}
\ncom{\cW}{{\mathcal W}}
\ncom{\cL}{{\mathcal L}}
\ncom{\cP}{{\mathcal P}}
\ncom{\cH}{{\mathcal H}}
\ncom{\cS}{{\mathcal S}}
\ncom{\cM}{{\mathcal M}}
\ncom{\cC}{{\mathcal C}}
\ncom{\cK}{{\mathcal K}}
\ncom{\cT}{{\mathcal T}}
\ncom{\cF}{{\mathcal F}}
\ncom{\cN}{{\mathcal N}}
\ncom{\cJ}{{\mathcal J}}
\ncom{\cV}{{\mathcal V}}
\ncom{\cZ}{{\mathcal Z}}
\ncom{\cU}{{\mathcal U}}
\ncom{\cSU}{{\mathcal S \mathcal U}}
\ncom{\cG}{{\mathcal G}}
\ncom{\cQ}{{\mathcal Q}}
\ncom{\cR}{{\mathcal R}}
\ncom{\cY}{{\mathcal Y}}
\ncom{\cE}{{\mathcal E}}
\ncom{\cI}{{\mathcal I}}
\ncom{\mylabel}[1]{{\rm (#1)}\label{#1}}
\ncom{\Hom}{{\textit{Hom}}}
\ncom{\eop}{{\hfill $\Box$}}
\begin{document}
\baselineskip=16pt

\title[Cohomological invariants of flat connections]{Cohomological invariants of a variation of flat connections}
\author{Jaya NN Iyer}

\address{The Institute of Mathematical Sciences, CIT
Campus, Taramani, Chennai 600113, India}
\email{jniyer@imsc.res.in}

\footnotetext{Mathematics Classification Number: 53C55, 53C07, 53C29, 53.50. }

\begin{abstract} In this paper, we apply the theory of Chern-Cheeger-Simons to construct canonical invariants associated to an $r$-simplex whose points parametrize flat connections on a smooth manifold $X$. These invariants lie in degrees $(2p-r-1)$-cohomology with $\C/\Z$-coefficients, for $p> r\geq 1$. This corresponds to a homomorphism on the higher homology groups of the moduli space of flat connections, and taking values in $\C/\Z$-cohomology of the underlying smooth manifold $X$. 
\end{abstract}

\maketitle

%\setcounter{tocdepth}{1}
%\tableofcontents

%%%%%%%%%%%%%%%%%%%%%%%%%%%%%%%%%%%%%%%%%%%%%%%%%%%%%%%%%%%%%%%%%%%%%%%%%%%%%%%%
\section{Introduction}
%%%%%%%%%%%%%%%%%%%%%%%%%%%%%%%%%%%%%%%%%%%%%%%%%%%%%%%%%%%%%%%%%%%%%%%%%%%%%%%

Suppose $X$ is a smooth manifold and $E$ is a smooth complex vector bundle on $X$ of rank $n$. 
Consider a smooth connection $\nabla$ on $E$. The Chern-Weil theory defines the Chern forms $P_p(\nabla^2,...,\nabla^2)$, of degree $2p$ and for $p\geq 1$. Here $P_p$ denotes a ${GL}_n$-invariant polynomial 
of degree $p$ such that $P_p(\nabla^2,...,\nabla^2)=tr(\nabla^2)$ and $\nabla^2$ denotes the curvature form 
of the connection form $\nabla$ (see \S \ref{convention}). The Chern forms are closed and correspond to the de Rham Chern classes $c_p(E)\in H^{2p}_{dR}(X,\C)$. These are the primary invariants of $E$, and are independent
of the connection. In particular when $\nabla$ is flat, i.e., when 
$\nabla^2=0$, the Chern-Cheeger-Simons 
theory \cite{Chn-Sm}, \cite{Ch-Sm} defines cohomology classes 
$$
\hat{c_p}(E,\nabla) \in H^{2p-1}(X,\C/\Z(p)),
$$
for $p\geq 1$.
These are the secondary invariants of $(E,\nabla)$. These invariants are known to be rigid in a family of flat connections, when $p\geq 2$ \cite[Proposition 2.9, p.61]{Ch-Sm}. In other words, the cohomology 
class remains the same in a variation of flat connections.

The proof of the rigidity is obtained by looking at the \text{eta}-differential 
form :
\begin{equation}\label{etainv}
\eta_p:= p.\int_I P_p(\f{d}{dt}\nabla_t, \nabla_t^2,...,\nabla_t^2). dt.
\end{equation}

 The difference of 
the Chern-Simons differential character $\hat{c_p}(E,\nabla_1)-\hat{c_p}(E,\nabla_0)$ is just the linear functional defined by the degree $(2p-1)$-form $\eta_p$ (see \cite[Proposition 2.9, p.61]{Ch-Sm}). This form also gives the rigidity of the Chern-Simons classes for flat connections, in degrees at least two.

M. Kontsevich  asked whether the  eta-form is related to certain liftings in Chow theory. 
Later, C. Simpson remarked that, these should provide maps  on the higher homotopy groups of the moduli space $\cM_X(n)$ of flat connections.

We investigate these remarks and thoughts in this paper.
We are motivated by Karoubi's simplicial invariants in \cite[p.68]{Karoubi}, where he considers sequences of smooth connections on a smooth vector bundle. His construction provides maps on the homology groups of moduli space of connections and takes values in the cohomology of the underlying manifold. The maps are provided by transgression of Chern-Weil forms and using integration along fibres of differential forms.

We would like to extend his methods, to obtain suitable maps on homology of moduli space of flat connections, by applying the Chern-Simons theory and integration along fibres.

Integration along fibres on differential cohomology have been defined by several authors, see \cite{Freed}, \cite{Singer}. In our situation, we need to apply for the projection $X\times \Delta^r\rar X$, $r\geq 1$. However $X\times \Delta^r$ is a manifold with boundary and corners. C. Baer and C. Becker (see \cite{BB14}) have defined 'integration along fibres' on differential cohomology for manifolds with boundary and C. Becker has adapted their proof to the projection
$X\times \Delta^r \rar X$. 
These constructions however seem inadequate, for our purpose. Hence we restrict the range on degrees where the pushforward on odd-degree $\C/\Z$-cohomology could be applied. We make this precise as follows.

In \S \ref{simplicialDE}, we first consider the simplicial set $\cD(E)$ associated to a smooth vector bundle $E$ on $X$. Denote $q:\Delta^r \times X\rar X$ the projection onto the second factor $X$, where $\Delta^r$ is the standard $r$-real simplex.
Consider a general connection $\tilde{D}$ on $q^*E$ over $\Delta^r\times X$.  These may be organized into a simplicial set $\cD(E)$. Namely, the $r$-simplices of this simplicial set are connections over $\Delta^r\times X$, and the face maps are given by restriction (degeneracies given by pullback).

Denote $\cZ_r(\cD(E))$ the cycles in this simplicial set, i.e. formal sums of simplices whose boundary is zero.

Using methods similar to \cite[\S 3]{Karoubi}, we deduce the following.

\begin{theorem} There is a non-zero homomorphism on the group of cycles $\cZ_r(\cD(E))$ of $\cD(E)$, for $r\geq 1$:
\begin{equation}\label{hommap0}
\rho_r:\mathcal{Z}_r(\cD(E)) \rar \bigoplus_{p>r} \widehat{H^{2p-r}}(X).
\end{equation}
\end{theorem}

When the points of an $r$-simplex parametrize flat connections, then the above construction yields a well-defined Chern-Simons class, see Proposition \ref{flatpropn}. 
We apply the pushforward map  on the $\C/\Z$-cohomology of $X\times \Delta^r$, to obtain well-defined classes, 
$$
CS(\tilde{D})_{p,r}\,\in\, H^{2p-r-1}(X,\C/\Z(p))
$$
when $r< p$.

In turn, the homomorphism in \eqref{hommap0} restricts and descends on the homology groups $\HH_r(\cM_X(n))$ of moduli space $\cM_X(n)$ of rank $n$ flat connections, and takes value in  $\C/\Z$-cohomology. More precisely, we have:
 
\begin{theorem}
There is a non-zero homomorphism, for $r\geq 1$:
$$
\rho'_r:\,\HH_r(\cM_X(n)) \rar \bigoplus_{p\,>\, r}H^{2p-r-1}(X, \C/\Z(p)).
$$
\end{theorem}
We briefly explain the construction of the maps $\rho_r $ and $\rho_r'$.

We consider a smooth connection $\tilde{D}:= t_0D_0+...+ t_rD_r$, with $\sum t_i=1$ and $t_i\geq 0$, on $q^*E$. Here $q:X\times \Delta^r \rar X$ is the first projection and $\Delta^r$ is an $r$-simplex with vertices $D_0,...,D_r$.

To define the map $\rho_r$, we consider the integration of Chern forms, denoted by 
$$
TP(\tilde{D})_{p,r}:= \int_{\Delta^r} P_p(\tilde{D}^2,...,\tilde{D}^2).
$$
This is a $(2p-r)$-differential form on $X$.
We now want to obtain a well-defined canonical lift of this form, as a differential character on $X$.

 We start by defining a
$(2p-r-1)$-form $\omega(\tilde{D})_{p,r}$ on $X$, and which satisfies the identity:
\begin{equation*}
d\omega(\tilde{D})_{p,r}= -\sum_{i=0}^r (-1)^r\omega(\tilde{D}_i)_{p,r-1}  -(-1)^{2p-r-1} TP(\tilde{D})_{p,r}.
\end{equation*}
See \eqref{Oform}.

This identity helps us to obtain a canonical lift of $TP(\tilde{D})_{p,r}$  as a well-defined differential character on $X$, if we replace $\tilde{D}$ by a cycle in $\cD(E)$. 
However, to obtain a map on homology of moduli space of flat connections, we note that the simplicial set $\cD(E)$  approximates the moduli space of smooth connections, and the above computations are carried over a sub-simplicial set associated to relatively flat connections.

The details are provided in \S \ref{mapHH}. 

The above constructions could also be carried out replacing the Chern forms by $GL_n$-invariant polynomial forms.

In the final section \S \ref{Rquestion}, we extend the Cheeger-Simons question on rationality of above maps.

{\Small Acknowledgement: This work is motivated by a question of M. Kontsevich on the rigidity property of the original Chern-Simons classes, arising from the above form $\eta_p$. We thank him for this question. Thanks also to C. Simpson for sharing his insightful and inspiring remarks, and motivating us in this direction.
The referee made helpful suggestions and important remarks (see Remark \ref{referee}), and we are grateful to him/her.

 We also thank M. Karoubi for pointing out to his works relevant to our constructions and to P. Deligne for his helpful remarks (especially on Prop.\ref{flatpropn}). Finally, we thank Scott Wilson, C. Baer and C. Becker for sharing their thoughts/works on fiber integration and comments. This work was initiated at IHES, Paris, during our stay in 2013 and we are grateful to the institute for their hospitality and support.}

%%%%%%%%%%%%%%%%%%%%%%%%%%%%%%%%%%%%%%%%%%%%%%%%%%%%%%%%%%%%%%%%%%%%%%%%%%%%%%%%%%%%%%%%%%%%%%%%%%%%%%%%%%%%%%%%
\section{Preliminaries: original Chern-Cheeger-Simons theory}
%%%%%%%%%%%%%%%%%%%%%%%%%%%%%%%%%%%%%%%%%%%%%%%%%%%%%%%%%%%%%%%%%%%%%%%%%%%%%%%%%%%%%%%%%%%%%%%%%%%%%%%%%%%%%%%

Suppose $X$ is a smooth (finite dimensional) manifold, and $E$ is a smooth complex vector bundle on $X$. Our aim is to construct a canonical lifting of the differential form \eqref{etainv}, 
in the differential 
cohomology, defined by Cheeger and Simons in \cite{Ch-Sm}. We start by recalling the original constructions.

%%%%%%%%%%%%%%%%%%%%%%%%%%%%%%%%%%%%%%%%%%%%%%%%%%%%%%%%%%
\subsection{Conventions}\label{convention}

 Denote by $\Z(p)$ the subgroup of $\C$ generated by $(2\pi i)^p$.
For each subgroup $L$ of $\C$, set
$$
L(p)=L\otimes_\Z \Z(p).
$$
The isomorphism $\Z\rar \Z(p)$ that takes $1$ to $(2\pi i)^p$ induces a canonical isomorphism
\begin{equation}\label{isomorphism}
H^\bullet(X,\Z)\rar H^\bullet(X,\Z(p)).
\end{equation}
The $p$-th Chern class of a complex vector bundle lies in $H^{2p}(X,\Z(p))$ which is the image of the usual topological Chern classes, under the isomorphism given by \eqref{isomorphism}.

Suppose $(E,\nabla)$ is a complex vector bundle of rank $n$ with a connection on $X$.
 Then the Chern forms
$$
c_k(E,\nabla) \in A^{2k}_{cl}(X,\Z)
$$
for $0\leq k\leq \m{rank }(E)$, are defined using the universal Weil homomorphism \cite[p.50]{Chn-Sm}.
There is a $GL_n$-invariant, symmetric, homogeneous and multilinear 
polynomial $P_k$ of degree $k$ in $k$ variables on the Lie algebra ${\textbf gl}_n$ such that if
$\Omega$ is the curvature of $\nabla$ then $c_k(E,\nabla ) = \cP_k (\f{-1}{2\pi i}\Omega , \ldots , \f{-1}{2\pi i}\Omega )$.
Here $P_k$ are defined as follows;
$$
\m{det}(I_r+x)=1+P_1(x)+P_2(x)+...+P_r(x),\,\,x\in \textbf{gl}_n(\C).
$$
When $x_i=x$ for each $i$, then $P_k(x,...,x)= \m{trace}(\wedge^k x)$ (see \cite[p.403]{GriffithsHarris}),
however the wedge product here is taken in the variable $\C ^n$, not the wedge of forms on the base. 
If $x$ is a diagonal matrix with eigenvalues $\lambda _1,\ldots , \lambda _r$ then
 $P_k(x,...,x) = \sum _I \lambda _{i_1}\cdots \lambda _{i_k}$. We can also express $P_k$ in terms
of the traces of products of matrices. In this expression,
the highest order term of $P_k$ is the symmetrization of $Tr (x_1\cdots x_k)$ multiplied by a constant, the
lower order terms are symmetrization's of $Tr(x_1\cdots x_{i_1})Tr(\cdots ) \cdots Tr(x_{i_a+1}\cdots x_k)$, 
with suitable constant coefficients. 

\begin{remark}
The following constructions and proofs also hold if we replace the polynomial $P_p$ by any $GL_n$-invariant polynomial of degree $p$, in particular we could also take the polynomial which gives the $p$-th Chern character term.
\end{remark}

We recall the convention on forms, from \cite[p.50]{Chn-Sm}:
denote $\bigwedge^{k,l}(E)$, $k$-forms on $E$ taking values in $\textbf{gl}_n^{\otimes l}$. We have the usual exterior
differential 
$$
d: \bigwedge^{k,l}(E) \rar \bigwedge^{k+1,l}(E).
$$
If $\phi\in \bigwedge^{k,l}(E),\,\psi\in \bigwedge^{k',l'}(E)$, then define
$$
\phi\wedge \psi \in \bigwedge^{k+k',l+l'}(E),
$$
\begin{equation*}
\phi\wedge \psi(x_1,...,x_{k+k'})= \sum_{\pi \m{ shuffle }} \sigma(\pi)\phi(x_{\pi_1},...,x_{\pi_k}) \otimes \psi(x_{\pi_{k+1}},...,x_{\pi_{k+1}},...,x_{\pi_{k+k'}}).
\end{equation*}
If $P$ is a $GL_n$-invariant polynomial of degree $l$ and $\phi \in \bigwedge^{k,l}(E)$. Then $P(\phi)=P\circ \phi$ is a real 
valued $k$-form on $E$. The following identities hold:
\begin{eqnarray*}
P(\phi\wedge \psi\wedge \rho) & = & (-1)^{kk'}P(\psi\wedge\phi\wedge\rho), \\
d(P(\phi)) & = & P(d\phi), \\
P(\phi\wedge \psi \wedge \rho) &= & (-1)^{kk'}P(\psi\wedge \phi \wedge \rho).
\end{eqnarray*}
Here $\phi \in \bigwedge^{k,l}, \psi \in \bigwedge^{k',l'}, \rho\in \bigwedge^{k",l"}$.

%%%%%%%%%%%%%%%%%%%%%%%%%%%%%%%%%%%%%%%%%%%%%%%%%%%%%%%
\subsection{Preliminaries}\cite{Ch-Sm}:

Suppose $X$ is a smooth manifold and $(E,\nabla)$ is a flat connection of rank $n$ on $X$. Recall the differential cohomology:
$$
\widehat{H^{p}}(X):= \{ (f,\alpha): f: Z_{p-1}(X)\rar \C/\Z,\, \partial(f)= \alpha, \alpha \m{ is a closed form and integral valued} \}.
$$
Here $Z_{p-1}(X)$ is the group of $(p-1)$-dimensional cycles, and $\partial$ is the coboundary map on cochains. The linear functional $\partial(f)$ is defined by integrating the form 
$\alpha$ against $p$-cycles, and it takes integral values on integral cycles.

In \cite{Ch-Sm}, the functional $f$ is $\R/\Z$-valued. We could also take $\C/\Z$-valued functionals, since we hope that our applications will utilise $\C/\Z$-characters.

This cohomology fits in an exact sequence:
\begin{equation}\label{diffcoh}
0\rar H^{p-1}(X,\C/\Z) \rar \widehat{H^{p}}(X) \rar A^{p}_\Z(X)_{cl} \rar 0.
\end{equation}

Here $A^{p}_\Z(X)_{cl}$ denotes the group of closed complex valued $p$-forms with integral periods.

Given a smooth connection $(E,\nabla)$ on $X$. The Chern forms $c_p(E,\nabla)$ are degree $2p$-differential forms on $X$.
Let $P_p$ denote the degree $p$ polynomial defining the $p$-Chern form, i.e., $c_p(E,\nabla):=P_p(\nabla^2,...,\nabla^2)$ (see \S \ref{convention}). Here $\nabla^2$ denotes the curvature form, and we say that $\nabla$ is \textit{flat} if $\nabla^2$ is identically zero.
 
We recall the following construction, to motivate the construction of tertiary classes in the next section.
\begin{theorem}\cite{Ch-Sm}\label{canDC}
The Cheeger-Simons construction defines a differential character, for any smooth connection $(E,\nabla)$:
$$
\hat{c_p}(E,\nabla)\, \in\, \widehat{H^{2p}}(X), \,\,\,p\geq 0.
$$
Moreover, if $\nabla$ is flat then 
$$
\hat{c_p}(E,\nabla)\, \in\, H^{2p-1}(X,\C/\Z(p)).
$$
This is the same as the Chern-Simons class $CS_{p}(E,\nabla)$.
\end{theorem}
\begin{proof}
Let $P_p$ denote the degree $p$, $GL_n$-invariant polynomial, defined in \S \ref{convention}. 
The curvature form of $\nabla$ is $\Theta:=\nabla^2$. The Chern form in degree $2p$ is given as $P_p(\Theta,...,\Theta)$. The Chern-Weil theory says that the Chern form is closed and that it has integral periods. Hence it gives an element in the group $A^{2p}_\Z(X)_{cl}$.
Now use the universal smooth connection \cite{Narasimhan}, which lies on a Grassmannian $G(r,N)$, for large $N$. The differential cohomology of $G(r,N)$ is just the group $A^{2p}_\Z(X)_{cl}$, since the 
odd degree cohomologies of $G(r,N)$ are zero.  Hence the differential character $\hat{c_p}$ is the Chern form of the universal connection. The pullback of this element via a classifying map of $(E,\nabla)$ defines the differential character
$\hat{c_p}(E,\nabla)\in \widehat{H^{2p}}(X)$. 

Suppose $\nabla$ is flat, i.e., $\Theta=0$. Then the Chern form $P_p(\Theta,...,\Theta)=0$. Hence the differential character lies in the group
$H^{2p-1}(X,\C/\Z(p))$.

\end{proof}

%%%%%%%%%%%%%%%%%%%%%%%%%%%%%%%%%%%%%%%%%%%%%%%%%%%%%%%%%%%%%%%%%%%%%%%%%%%%%%%%%%%%%%%%%%%%%%%%%
\section{Pushforward of Chern-Simons forms}\label{analytichigherCS}

%%%%%%%%%%%%%%%%%%%%%%%%%%%%%%%%%%%%%%%%%%%%%%%%%%%%%%%%%%%%%%%%%%%%%%%%%%%%%%%%%%%

\subsection{Rigidity of Chern-Cheeger-Simons classes and the eta-form}

As in the previous section, suppose $(E,\nabla)$ is a smooth connection on a smooth manifold. Suppose
$\gamma:=\{\nabla_t\}_{t\in I}$ is a one parameter family of smooth connections on $E$ and $\nabla=\nabla_0$. Here $I:=[0,1]$ is the closed unit interval.

The family $\{\nabla_t\}_{t\in I}$ gives us a family of differential characters:
$$
\hat{c_p}(E,\nabla_t) \,\in\, \widehat{H^{2p}}(X)
$$
for $p\geq 1$ and $t\in [0,1]$.

The difference of the differential characters when $t=0$ and when $t=1$, is given by the following variational formula.

\begin{proposition}
With notations as above, we have the following equality:
$$
\hat{c_p}(E,\nabla_1)-\hat{c_p}(E,\nabla_0) \,=\, p. \int_{0}^1 P_p(\f{d}{dt}\nabla_t,\nabla_t^2,...,\nabla_t^{2}) dt
$$
for $p\geq 1$.
\end{proposition}
\begin{proof}
See \cite[p.61, Proposition 2.9]{Ch-Sm}.
\end{proof}

In particular, if $\{\nabla_t\}_{t\in I}$ is a family of flat connections then the degree $(2p-1)$-form 
\begin{equation}\label{etapform}
\eta_p:=  p. \int_{0}^1 P_p(\f{d}{dt}\nabla_t ,\nabla_t^{2},..., \nabla_t^2)dt
\end{equation}
is identically zero, when $p\geq 2$. This implies that
$$
\hat{c_p}(E,\nabla_1)\,=\,\hat{c_p}(E,\nabla_0).
$$
This gives the \textit{rigidity} of Chern-Simons classes
\begin{equation}\label{rigidCS}
CS_p(E,\nabla_1)\,=\,CS_p(E,\nabla_0)\,\in\, H^{2p-1}(X,\C/\Z)
\end{equation}
whenever $p\geq 2$.

%%%%%%%%%%%%%%%%%%%%%%%%%%%%%%%%%%%%%%%%%%%%%%%%%%%%%%%%%%%%%%%%%%%%%%%%%%%%%%%
\subsection{Canonical invariants lifting the difference of Chern-Simons classes}

We motivate our constructions by looking at the following situation.

Recall the coefficient sequence:
$$
0\rar \Z(p) \rar \C \rar \C/\Z(p) \rar 0.
$$
The associated long exact cohomology sequence is given by:
$$
\rar H^k(X,\C/\Z(p)) \rar H^{k+1}(X,\Z(p)) \rar H^{k+1}(X,\C) \rar H^{k+1}(X, \C/\Z(p)) \rar.
$$
Suppose $k+1=2p$, and $(E,\nabla)$ is a flat connection then the Chern-Simons class $CS_p(E,\nabla)$ is a canonical lifting of the (torsion) integral Chern class $c_p(E)$. Here we note that the de Rham Chern class
$c_p(E,\nabla)\in H^{2p}(X,\C)$ is identically zero, for $p\geq 1$.
We noticed in  proof of Theorem \ref{canDC}, that the differential cohomology is introduced as an intermediate object in the above cohomology sequence where the canonical invariants lie and they lift the de Rham Chern form.

Suppose $k+1=2p-1$, and we are given a path of flat connection $\gamma:=\{\nabla_t\}_{t\in I}$.  
The difference of Chern-Simons classes, using the rigidity formula in \eqref{rigidCS}, gives us
$$
CS_p(E,\nabla_1)\,-\,CS_p(E,\nabla_0)\,=\,0\,\in\, H^{2p-1}(X,\C/\Z(p)).
$$

This difference is characterised by the eta-form $\eta_p$, in \eqref{etapform}.
Hence, we would like to construct a canonical lifting in $H^{2p-2}(X,\C/\Z(p))$, via differential characters. The differential character should be a canonical lifting of the eta form $\eta_p$, upto an exact form.
We do not write a formula for the lift but rather consider suitable lifting of closely associated forms $TP_p$, which we will define in the next section.

%%%%%%%%%%%%%%%%%%%%%%%%%%%%%%%%%%%%%%%%%%%%%%%
\subsection{Integration along fibres}

Suppose $\pi:X\times T \rar X$ and $E$ is a vector bundle on $X$. Suppose $\tilde{\Theta}$ is a smooth connection on $\pi^*E$, $T$ is a smooth compact manifold of dimension $r$ or the standard $r$-simplex $\Delta^r$. Then there is a  map given by integration along the fibres
\begin{equation}\label{homotopyform}
B:A^{k}(X\times T)\rar A^{k-r}(X),\, \alpha \mapsto \int_T \alpha.
\end{equation}
 Similarly, there is a map (given by proper pushforward)
\begin{equation}\label{homcoh}
B: H^k(X\times T,\C/\Z) \rar H^{k-r}(X, \C/\Z).
\end{equation}

We note that the above two maps in  \eqref{homotopyform}, \eqref{homcoh} give a differential character $\hat{B}(\hat{f})$ on $X$ for a  differential character $\hat{f}$ on $X\times T$, under some conditions on $T$.
See \cite{BB14}, \cite[Proposition 2.1]{Freed}. 

However, in our situation $T$ is an $r$-simplex and $X\times \Delta^r$ is a manifold with boundary and corners, and the  maps on differential cohomology defined by above authors seem inadequate. 
 Hence, we restrict to an appropriate range where the Chern-forms on $X\times \Delta^r$ vanish identically, and we need to only apply the map in \eqref{homcoh}.

%%%%%%%%%%%%%%%%%%%%%%%%%%%%%%%%%%%%%%%%%%%%%%%%%%%%%%%%%%%%%%
\subsection{Stokes' theorem}

We recall Stokes' theorem, which we will use in proof of Theorem \ref{higherhomotopy}.

Suppose $\pi: X\times T\rar X$ is the first projection and $r=dim(T)$ (as in previous section). If $\eta$ is a degree $q$-differential form on $X\times T$ then Stokes' theorem asserts the following:
\begin{equation}\label{stokes}
d\,\int_{T} \eta \,=\, \int_{T} d\eta - (-1)^{q-r}\int_{\partial T}\eta.
\end{equation}

%%%%%%%%%%%%%%%%%%%%%%%%%%%%%%%%%%%%%%%%%%%%%%%%%%%%%%%%%%%%%%%%%%%%%%%%%%%%%%%%%%%%%%%%%%%%%%%%%%%%%%%%%%%%%%%%%%%%%%%%%%%%%%%%

%%%%%%%%%%%%%%%%%%%%%%%%%%%%%%%%%%%%%%%%%%%%%%%%%%%%%%%%%%%%%%%%%%%
\section{(Co)Homological invariants of the simplicial set of flat connections}
%%%%%%%%%%%%%%%%%%%%%%%%%%%%%%%%%%%%%%%%%%%%%%%%%%%%%%%%%%%%%%%%%

%%%%%%%%%%%%%%%%%%%%%%%%%%%%%%%%%%%%%%%%%%%%%%%%%%%%%%%%%%%%%%%%%%%%%%%%%%%%%%%%%%%%%%
\subsection{Simplicial set $\cD(E)$ of sequence of flat connections}\label{simplicialDE}

We start with the approach of M. Karoubi \cite{Karoubi}.

Now suppose we form the simplicial set $\cD(E)$ formed by sequences of \textit{flat} connections, under the Gauge group transformation. We would like to understand the topology of $\cD(E)$, in terms of giving suitable maps on homology groups of $\cD(E)$.

Given $E$ on $X$, we associate a simplicial abelian group $M_*$: let $M_r$ be the free abelian group generated by the sequences $(D^0,D^1,...,D^r)$ where $D^i$ are smooth flat connections on $E$.
If $\phi: [s]\rar [r]$ is an increasing function, the map
$$
\phi^*:M_r \rar M_s
$$
is defined by 
$$
\phi^*(D^0,...,D^r)=(\nabla^0,\nabla^1,...,\nabla^s)
$$
with $\nabla^i:=D^{\phi(i)}$.
The gauge group $\G$ of automorphisms of $E$ acts on $M_*$ as follows:
$$
\phi^*(D^0,...,D^r).g= (g^*D^0,g^*D^1,...,g^*D^r).
$$
The quotient $N_*:=M_*/\G$ is also a simplical abelian group.
We denote $\HH_r(E)$ the homology groups of $N_*$.

We pose the following question:
\begin{question}\label{higherhom}
Is there a non-zero group homomorphism, for each $r\geq 1$ :
$$
\rho_{p,r}: \HH_r(\cD(E)) \rar \widehat{H^{2p-r}}(X).
$$
\end{question}

This question is motivated by the following computations.

Let $\Delta^r$ be the standard $r$-simplex with the variables $t=(t_0,t_1,...,t_r)$ such that $t=\sum_{i=0}^rt_i$. Let $\tilde{D}$ be the connection 
$$
\tilde{D}:= t_0D^0+t_1D^1+...+ t_rD^r
$$
on $q^*E$, where $q:\Delta^r\times X\rar X$ is the second projection.

For $r\geq 1$ and $r\leq p$, denote
$$
TP(\tilde{D})_{p,r}:= \int_{\Delta^r} P_p(\tilde{D}^2,...,\tilde{D}^2).
$$
This is a differential form of degree $2p-r$ on $X$.

We recall the following formula (see \cite[Theoreme 3.3,p.67]{Karoubi}):
\begin{equation}\label{Kform}
dTP(\tilde{D})_{p,r}= -\sum_{i=0}^r (-1)^i TP(\tilde{D}_i)_{p,r-1}.
\end{equation}

When $r=1$, then we have
$$
dTP(\tilde{D})_{p,1}= c_p(D^1)- c_p(D^0).
$$
This is the difference of Chern forms.  Since $D^0$ and $D^1$ are flat connections, the Chern forms are zero.   In other words, $TP(\tilde{D})_{p,1}$ is a closed form. 

When $r>1$, it is not immediate that the form $TP_{p,r}$ is a closed form, unless all $TP_{p,r-1}$ are zero or cancel out (using \eqref{Kform}). To rectify this situation, we define a canonical differential character lifting $TP(\tilde{D})_{p,r}$, in the next subsection.

\begin{remark}\label{referee}
Referee's remarks: It is a very interesting idea that we could approximate a general family of connections by simplicial connections of the form given above,i.e. of the type:
let $\Delta^r$ be the standard $r$-simplex with the variables $t=(t_0,t_1,...,t_r)$ such that $t=\sum_{i=0}^rt_i$. Let $\tilde{D}$ be the connection 
$$
\tilde{D}:= t_0D^1+t_1D^1+...+ t_rD^r
$$
on $q^*E$. Here $q:\Delta^r\times X\rar X$ is the second projection.
If the $D^j$ are flat there is no guarantee that their linear combinatios will be flat, and poses some problems in applying it directly in Theorem \ref{mainthm}. If they are flat and nearby, then the linear combinations will be approximately flat, so this should give a nice way of approximating some integrals. 
\end{remark}

We now proceed to answer Question \ref{higherhom}, partly.

We consider a general connection $\tilde{D}$ on $q^*E$ over $\Delta^r\times X$.  These may be organized into a simplicial set $\cD(E)$. Namely, the $r$-simplices of this simplicial set are connections over $\Delta^r\times X$, and the face maps are given by restriction (degeneracies given by pullback).

We denote $\tilde{D_i}$ the restriction of $\tilde{D}$ on the $i$-th face of  $\Delta^r\times X$. We consider this simplicial set, so as to directly  obtain the required map on homology of moduli space of flat connections, in Theorem \ref{mainthm}.

We note that the formula in \eqref{Kform}, 
\begin{equation*}
dTP(\tilde{D})_{p,r}= -\sum_{i=0}^r (-1)^i TP(\tilde{D}_i)_{p,r-1}
\end{equation*}
holds for a general connection $\tilde{D}$ over $\Delta^r \times X$. The computation given in \cite[Theoreme 3.3, p.67]{Karoubi} remains valid for $\tilde{D}$, which is an application of Stokes' theorem.

%%%%%%%%%%%%%%%%%%%%%%%%%%%%%%%%%%%%%%%%%%%%%%%%%%%
\subsection{Higher homotopy maps}

Denote $\mathcal{Z}_r(\cD(E))$ the group of  $r$-cycles of the simplicial set $\cD(E)$, i.e., to say formal sums of simplices whose boundary is zero.

Then we have the following:

\begin{theorem}\label{higherhomotopy}
With notations as in previous subsection, and
when $p>r\geq 1$, there are well-defined maps
$$
\rho_{p,r}:\mathcal{Z}_r(\mathcal{D}(E))\rar \widehat{H^{2p-r}}(X).
$$

\end{theorem}
\begin{proof}

With notations as in \S \ref{simplicialDE}, 
we define the 'fiber integration' element of the differential character of $(q^*E,\tilde{D})$, on $ X$:
$$
\widehat{TP(\tilde{D})}_{p,r}= (tp(\tilde{D})_{p,r}, TP(\tilde{D})_{p,r})
$$
This element is well-defined if $tp(\tilde{D})_{p,r}: Z_{2p-r-1}(X)\rar \C/\Z(p) $ is a degree $(2p-r-1)$ cochain whose coboundary is the cochain corresponding to the degree $(2p-r)$ form $TP(\tilde{D})_{p,r}$. In particular, we will show that if $\sigma$ is a cycle in $\cD(E)$, then the differential character $\widehat{TP(\sigma)}$ lifting the associated differential form $TP(\sigma)$, is well-defined.

For this purpose, we now write
$$
tp(\tilde{D})_{p,r}: Z_{2p-r-1}(X)\rar \C/\Z(p)
$$
as the map $c' \mapsto \int_{c'} \omega(\tilde{D})_{p,r}$, for a canonical form $\omega(\tilde{D})_{p,r}\in A^{2p-r-1}(X)$, such that $d\omega(\tilde{D})_{p,r}= TP(\tilde{D})_{p,r}$ (upto certain forms depending on the $r-1$-faces of $\Delta^r$). 

We proceed to define the form $\omega(\tilde{D})_{p,r}$.

When $r=1$, we can make this explicit by choosing a path $\psi_s:=\{\tilde{D}^{s}\}$ of $1$-simplices associated to  length one sequences, such that when $s=1$, $\psi_1=\tilde{D}$ and when $s=0$, $\psi_0= \tilde{D}^0$. Here $\tilde{D}^{0}$ is the trivial $1$-simplex.
 Then define, when $p>1$:
$$
\omega(\tilde{D})_{p,1} := \int_I TP(\tilde{D}^{s})_{p,1}=\int_I \int_{I} P_p(\tilde{D}^{s},...,\tilde{D}^{s}).
$$
By Stokes' theorem (see \eqref{stokes}), we have
$$
d\omega(\tilde{D})_{p,1}= TP(\tilde{D})_{p,1}- TP(\tilde{D}^{0})_{p,1} = TP(\tilde{D})_{p,1}.
$$

In fact  when $p>r\geq 1$, we similarly define:
$$
\omega(\tilde{D})_{p,r}:= \int_{I} TP(\tilde{D}^{s})_{p,r}
$$
where $\psi_s:=\{\tilde{D}^{s}\}$ is a path of $r$-simplices, where $\psi_1=\tilde{D}$ and $\psi_0$ is the trivial $r$-simplex.

Recall that the $i$-th face of the $r$-simplex $\tilde{D}^{s}$ is denoted by $\tilde{D}^{s}_{i}$, which is an $(r-1)$-simplex. Then notice, by Stokes' theorem \eqref{stokes},
\begin{eqnarray*}
d\omega(\tilde{D})_{p,r}  & = & \int_{ I} dTP(\tilde{D}^{s})_{p,r}- (-1)^{2p-r-1} [TP(\tilde{D})_{p,r}-0]\\
                                & = & - \int_I \sum_{i=0}^r(-1)^i TP(\tilde{D}^{s}_i)_{p,r-1} - (-1)^{2p-r-1}TP(\tilde{D)}_{p,r}, \m{ use }\eqref{Kform}\\
                                & =& -\sum_{i=0}^{r}(-1)^i\int_I TP(\tilde{D}^{s}_{i})_{p,r-1} -(-1)^{2p-r-1} TP(\tilde{D})_{p,r}  .\\
\end{eqnarray*}

Hence the following identity  holds: 
\begin{equation}\label{Oform}
d\omega(\tilde{D})_{p,r}= -\sum_{i=0}^r (-1)^i\omega(\tilde{D}_i)_{p,r-1}  -(-1)^{2p-r-1} TP(\tilde{D})_{p,r}.
\end{equation}

Hence, given a $(2p-r-1)$-cycle $c$ on $X$, if we denote 
$$
<c,\tilde{D}>\,:= -\int_c \omega(\tilde{D})_{p,r} 
$$
 then we have the equality:
\begin{eqnarray*}
<c,\partial \tilde{D}> &=& - \int_c \sum_{i=0}^{r}(-1)^i\omega(\tilde{D}_i)_{p,r-1} \\
                          &=& \int_c d\omega(\tilde{D})_{p,r} +(-1)^{2p-r-1} TP(\tilde{D})_{p,r} ,\m{ use } \eqref{Oform} \\
 & = & -\int_{\partial c} \omega(\tilde{D})_{p,r}  + \int_c  (-1)^{2p-r-1} TP(\tilde{D})_{p,r}  , \m{ by Stokes' theorem }   \\
  &= & -<\partial c,\tilde{D}>  + \int_c (-1)^{2p-r-1} TP(\tilde{D})_{p,r} .
 \end{eqnarray*}

Now replacing $\tilde{D}$ by a $\sigma$, which is a finite linear combination of $r$-simplices $\tilde{D}$,
we deduce:

1) if $\sigma$ is a cycle in $\cD(E)$ then the singular cochain $c\mapsto <c,\sigma>$ has coboundary 
$(-1)^{2p-r-1}TP(\sigma)_{p,r}$ on $X$. In other words, it corresponds to the differential character
$(<.,\sigma>, (-1)^{2p-r-1}TP(\sigma)_{p,r})$ (see following proof).

2) if $\sigma$ is a boundary $\sigma=\partial \sigma'$, then the singular cochain $c\mapsto <c,\sigma>$ is  equal to
\begin{eqnarray*}
<c,\partial\sigma'> & =& -<\partial c,\sigma'> +\int_c (-1)^{2p-r-1} TP( \sigma').
\end{eqnarray*}
If $c$ is a cycle and if the integral is zero then the quantity on the right hand side is zero. 
                               
Hence from 1), we obtain a group homomorphism on the group of $r$- cycles of the simplicial set $\cD(E)$ as follows:
 
 for $p>r\geq 1$, define
\begin{eqnarray*}
\rho_{p,r}: \mathcal{Z}_r(\cD(E)) & \rar & \widehat{H^{2p-r}}(X) \\
                  \sigma & \mapsto & (<.,\sigma>, (-1)^{2p-r-1}TP(\sigma)_{p,r}).\\
\end{eqnarray*}
This is well-defined since the coboundary on group of cochains $C^\bullet(X)$ on $X$,
$$
\partial: C^{2p-r-1}(X)\rar C^{2p-r}(X)
$$
fulfils, by 1): 
$$
\partial (<.,\sigma>)= \int_{-}(-1)^{2p-r-1}TP(\sigma)_{p,r}.
$$
Since the Chern forms $P_p(\tilde{D}^s,...,\tilde{D}^s)$ over $\Delta^r\times X$ have integral periods, the forms $TP(\tilde{D}^s)_{p,r}$ also have integral periods, and are closed forms when $\tilde{D}^s$ is replaced by a cycle $\sigma$. 

\end{proof}

To deduce a corresponding map on homology of moduli space of flat connections, we first note the following proposition. In particular, the integral in 2) above will vanish and the map $\rho_{p,r}$ will descend on the homology.

When $r<p$, it was pointed by P. Deligne that we could apply integration along  fibres, directly on $H^{2p-1}(\Delta^r\times X,\C/\Z(p))$, since the Chern forms of $\tilde{D}$ vanish in this range. We make this precise as follows.

\begin{proposition}\label{flatpropn}(Deligne)
 Assume for each fixed $t$ as above, the connection $\tilde{D}^t$ is a flat connection.
(In particular, this assumption will hold when  $\Delta^r$ is an $r$-simplex mapping into the moduli space $M_X(k)$ of rank $k$ flat connections on $X$.)
Then there are well-defined cohomology classes, for $r<p$:
$$
CS(\tilde{D})_{p,r} \in H^{2p-r-1}(X,\C/\Z(p)).
$$
\end{proposition}
\begin{proof} Consider the degree $p$, $GL_n$-invariant polynomial $P_p$ (see \S \ref{convention}).

By assumption, for  each $t\in \Delta^r$, $\tilde{D}_t$ is flat and hence for $p>r\geq 1$, the $p$-th Chern form
$$
P_p(\tilde{D}^2,...,\tilde{D}^2)=0. 
$$
Locally, this is easy to check since each term of this form will have a $dx\wedge dx$. 

Hence we can apply the pushforward map (see \eqref{homcoh}) on the Chern-Simons class 
$$
{CS_p}(q^*E,\tilde{D})\in H^{2p-1}(\Delta^r \times X,\C/\Z(p)),
$$
 to  obtain a well-defined class 
 $$
 CS(\tilde{D})_{p,r}\in H^{2p-r-1}(X, \C/\Z(p)).
 $$ 
\end{proof}

\begin{remark}\label{flatdescent} The hypothesis in above proposition is same as saying that a connnection $\tilde{D}$ over $\Delta^r \times X\rar \Delta^r$ is relatively flat. 
We note that in the above situation the form $TP(\tilde{D})_{p,r}$ is zero, in particular the integral in 2) (in proof of Theorem \ref{higherhomotopy}) is zero. This property will enable us to descend the restriction of the map $\rho_r$ in  Theorem \ref{higherhomotopy}, to the homology of  simplicial set  parametrizing flat connections.
\end{remark}

%%%%%%%%%%%%%%%%%%%%%%%%%%%%%%%%%%%%%%%%%%%%%%%%%%%%%%%%%%%%%%%%%%%%%%%%%%%
\section{Cohomological invariants of moduli space of flat connections }\label{mapHH}

%%%%%%%%%%%%%%%%%%%%%%%%%%%%%%%%%%%%%%%%%%%%%%%%%%%%%%%%%%%%%%%%%

Suppose $M_X(n)$ denotes the moduli space of flat connections of rank $n$, on $X$.

Here, $M_X(n)$ is the topological type of Artin stack (see \cite{Noohi}), of space of isomorphism classes of rank $n$  representations of the fundamental group of $X$. Equivalently, the cohomology in the following statements, will mean the equivariant cohomology of the representation space, for the conjugation action of the group by choice of frames. 

We recall the Chern-Simons theory on $X$ and the rigidity property of the Chern-Simons classes  in \eqref{rigidCS}.
The rigidity property of the  Chern-Simons classes can be thought of as giving a map:
$$
\pi_0(M_X(n)) \rar \bigoplus_{p\geq 2} H^{2p-1}(X, \C/\Z(p)),
$$
on the connected components of the moduli space $M_X(n)$.

We would like to express the construction of the higher-order invariants associated to a variation of flat connections (see previous section), as a map on the higher homology groups of the moduli space. 

As a consequence of Proposition \ref{flatpropn} and Theorem \ref{higherhomotopy}, we  obtain the following map, on the homology groups of $M_X(n)$.
\begin{theorem}\label{mainthm}
There is a group homomorphism, for $r\geq 1$:
$$
\rho'_{r}: \HH_r (M_X(n)) \rar \bigoplus_{p > r} H^{2p-r-1}(X, \C/\Z(p)).
$$
\end{theorem}

\begin{proof}
Consider the simplicial set $\mathcal{S}:=\mathcal{S}(M_X(n))$ associated to the moduli space $M_X(n)$.
This means, for each $r\geq 0$, we have
$$
\mathcal{S}_r:= \{ f: \Delta^r \rar M_X(n)\}
$$
together with the usual face operators $d_i$ and degeneracy operators $s_i$, for $0\leq i\leq n$, satisfying the simplicial identities.

We now relate the homology of $\cS$ with the homology of $\cD(E)$ from \S \ref{simplicialDE}.
If $\cD(E)_{fl}\subset \cD(E)$ denotes the sub-simplicial set of relatively flat connections $\tilde{D}$ on
$\Delta^\bullet \times X$ then there is an induced map
$$
\mathcal{Z}_r(\cD(E)_{fl}) \rar \mathcal{Z}_r(\cD(E))
$$
on the group of $r$-cycles of $\cD(E)_{fl}$.

Composed with the map in Theorem \ref{higherhomotopy}, we obtain a map:
$$
\rho_{p,r,fl}: \mathcal{Z}_r(\cD(E)_{fl}) \rar \widehat{H^{2p-r}}(X).
$$
Using 2) in proof of Theorem \ref{higherhomotopy}, we have:
if $\sigma$ is a coboundary in $\cD(E)_{fl}$, i.e. $\sigma=\partial \sigma'$, then the singular cochain $c\mapsto <c,\sigma>$ is  equal to
\begin{eqnarray}\label{descentmap}
<c,\partial\sigma'> & =& -<\partial c,\sigma'> +\int_c (-1)^{2p-r-1} TP( \sigma').
\end{eqnarray}
If $c$ is a cycle and if the integral is zero then the quantity on the right hand side is zero. 
            
            By Remark \ref{flatdescent}, when $r<p$, the form $TP(\tilde{D})_{p,r}$ is identically zero, for a relatively flat connection $\tilde{D}$ over $\Delta^r \times X \rar X$.
            This implies that for a coboundary $\sigma=\partial \sigma'$ in $\cD(E)_{fl}$, the associated form $TP(\sigma')$ is identically zero. This is same as saying that $\rho_{p,r,fl}(\sigma)=0$.
            Hence, by \eqref{descentmap}, 
we conclude that the map $\rho_{p,r,fl}$ descends on the homology $\HH_r(\cD(E)_{fl})$. 

For $f\in \cS_r$, there is a projection $\pi:\Delta^r\times X\rar X$,  together with  a universal connection $\tilde{\nabla}$ on $\pi^*E$, such that $\tilde{\nabla}_t$ is a flat connection for any $t\in \Delta^r$. 

 Using arbitrary connections and the associated simplicial set $\cD(E)$, we note that the map $\HH_r(\cD(E)_{fl}) \rar \HH_r(\cS)$ is an isomorphism, which is the same as the homology $\HH_r(M_X(n))$ of the moduli space $M_X(n)$. Hence,
 we deduce a map
\begin{equation}\label{hommap}
\rho'_r:=\oplus_{p>r}\rho_{p,r,fl}:\HH_r(M_X(n)) \rar \bigoplus_{p>r}\widehat{H^{2p-r}}(X).
\end{equation}

However, using  Proposition \ref{flatpropn} and Remark \ref{flatdescent}, we note that the $(2p-r)$-degree form
$$
TP(\tilde{D})_{p,r}\,=\, 0
$$
whenever $1\leq r <p$ and $\tilde{D}$ parametrises flat connections.

This implies that the map $\rho'_r$ in \eqref{hommap} actually maps into the $\C/\Z(p)$-cohomology, i.e., 
$\rho'_r$ factorises as follows:
$$
\rho'_r: \HH_r(M_X(n)) \rar \bigoplus_{p>r}H^{2p-r-1}(X,\C/\Z(p))
$$
whenever $1\leq r<p$.

\end{proof}

We refer to $\rho'_r$, as the higher-order Cheeger-Simons invariants, in degree $r$.

%%%%%%%%%%%%%%%%%%%%%%%%%%%%%%%%%%%%%%%%%%%%%%%%%%%%%%%%%%%%%%%%%%
\section{Questions on rationality of (higher) Cheeger-Simons invariants}\label{Rquestion}
%%%%%%%%%%%%%%%%%%%%%%%%%%%%%%%%%%%%%%%%%%%%%%%%%%%%%%%%%%%%%%%%

Recall, that when $r=0$, and $X$ is a smooth manifold we have the map:
$$
\rho_0: \pi_0 (M_X(n)) \rar H^{2p-1}(X, \C/\Z(p)).
$$
 Cheeger and Simons \cite{Ch-Sm}, raised the following question:

\textit{Question 1}: 
Since the moduli space of representations has countably many connected components, 
 are the classes
$$
\widehat{c_p}(E,\nabla)\,\in\, H^{2p-1}(X,\C/\Z(p))
$$
torsion ?  i.e, takes value in $\Q/\Z$, whenever $p\geq 2$.

For reasonable smooth manifolds $X$, the homology groups of $M_X(n)$ are countably generated, and hence the above question can be extended as follows:

\textit{Question 2}:
 When $r\geq 1$ and $r<p$, does the image of the map $\rho_r$ lie in the torsion subgroup of $\C/\Z(p)$-cohomology of $X$ ?

We note that when $X$ is a smooth complex projective variety, Question 1  (i.e., when $r=0$) was answered positively by A. Reznikov \cite{Rez}, and partial results in the quasi-projective case  was obtained by Simpson and the author \cite{Iy-Si}.

%%%%%%%%%%%%%%%%%%%%%%%%%%%%%%%%%%%%%%%%%%%%%%%%%%%%%%%%%%%%%%%%%%%%%%%%%%%%%%%%%%%%%%%%%%%%%%%%%%%%%%%%%%%%%%%%%%%
%%%%%%%%%%%%%%%%%%%%%%%%%%%%%%%%%%%%%%%%%%%%%%%%%%%%%%%%%%%%%%%%%%%%%%%%%%%%

%%%%%%%%%%%%%%%%%%%%%%%%%%%%%%%%%%%%%%%%%%%%%%%%%%%%%%%%%%%%%%%%%%%%%%%%%%%%%%%%%%%%%%%%%%%%%%%%%%%%%%%%%%%%%%%%%%%%%%%%%%%%%%


\begin{thebibliography}{AAAAA}

\bibitem[1]{BB14}
{\sc C.~B\"ar, C.~Becker:}
{\em Differential Characters and Geometric Chains};
in: Differential Characters, Lecture Notes in Math.~{\bf 2112}, Springer, Berlin 2014, p.~1--90

\bibitem[2]{Ch-Sm} J. Cheeger, J. Simons, {\em
Differential characters and geometric invariants}, Geometry and topology (College Park, Md., 1983/84), 
50--80, Lecture Notes in Math., \textbf{1167}, Springer, Berlin, 1985. 

\bibitem[3]{Chn-Sm} S.S. Chern, J. Simons, {\em Characteristic forms and geometric invariants}, The Annals of Mathematics, Second Series, Vol. 99, No. \textbf{1}, 1974. 

\bibitem[4]{DHZ} J. Dupont, R. Hain, S. Zucker, {\em Regulators and characteristic classes of flat bundles},  
The arithmetic and geometry of algebraic cycles (Banff, AB, 1998),  47--92, CRM Proc. Lecture Notes, \textbf{24}, Amer. Math. Soc., Providence, RI, 2000. .

\bibitem[5]{Freed}  D.S. Freed,  {\em Classical Chern-Simons theory. II.} Special issue for S. S. Chern. Houston J. Math. 28 (2002), no. \textbf{2}, 293--310. 

\bibitem[6]{GriffithsHarris} P. Griffiths, J. Harris, {\em Principles in Algebraic Geometry}. Reprint of the 1978 original. 
Wiley Classics Library. John Wiley and Sons, Inc., New York, 1994. xiv+813 pp.

\bibitem[7]{Singer} M. J. Hopkins, I. M. Singer, 
{\em Quadratic functions in geometry, topology, and M-theory.}
J. Differential Geom. 70 (2005), no. \textbf{3}, 329--452. 

\bibitem[8]{Iy-Si} J. N. Iyer, C. T. Simpson, {\em Regulators of canonical extensions are torsion; the smooth divisor case}, preprint 2007, arXiv math.AG/07070372.

\bibitem[9]{Karoubi} M. Karoubi,
{\em Th\'eorie g\'en\'erale des classes caract\'eristiques secondaires.} K-theory t. \textbf{4}, p. 55-87 (1990). 

\bibitem[10]{Narasimhan} M. S. Narasimhan, S. Ramanan, {\em Existence of universal connections}, I, II,  Amer. J. Math.  \textbf{83}  1961 563--572,  \textbf{85} (1963), 223--231. . 

\bibitem[11]{Noohi}  B. Noohi, {\em Fundamental groups of algebraic stacks. } J. Inst. Math. Jussieu 3 (2004), no. \textbf{1}, 69--103.


\bibitem[12]{Rez} A. Reznikov, {\em All regulators of flat bundles are torsion. } Ann. of Math. (2) 141 (1995), no. \textbf{2}, 373--386. 

\bibitem[13]{Sm-Su} J. Simons, D. Sullivan, {\em Axiomatic characterization of ordinary differential cohomology}. J. Topol. 1 (2008), no. \textbf{1}, 45--56.

\end {thebibliography}
\end{document}